\documentclass[12pt,english,leqno]{amsart}
\usepackage[T1]{fontenc}
\usepackage[latin9]{inputenc}
\usepackage{amsthm}
\usepackage{amstext}
\usepackage{amssymb}

\makeatletter
\numberwithin{equation}{section}
\numberwithin{figure}{section}
\theoremstyle{plain}
\newtheorem{thm}{Theorem}
  \theoremstyle{plain}
  \newtheorem{prop}[thm]{Proposition}
  \theoremstyle{definition}
  \newtheorem{problem}[thm]{Problem}
  \theoremstyle{plain}
  \newtheorem{cor}[thm]{Corollary}
  \theoremstyle{plain}
  \newtheorem{lem}[thm]{Lemma}
 \theoremstyle{definition}
  \newtheorem{example}[thm]{Example}
  \theoremstyle{remark}
  \newtheorem*{claim*}{Claim}

\newcommand{\inc}{\operatornamewithlimits{\subseteq}}

\newcommand{\union}{\operatornamewithlimits{\bigcup}}

\newcommand{\cl}{\operatorname{cl}}

\newcommand{\intr}{\operatorname{int}}
\newcommand{\adh}{\operatorname{adh}}

\newcommand{\lm}{\operatorname{lim}}

\newcommand{\Epi}{\operatorname{Epi}}
\newcommand{\card}{\operatorname{card}}

\def\A{{\mathcal A}}  \def\B{{\mathcal B}} \def\C{{\mathcal C}}
\def\D{{\mathcal D}}
 \def\E{{\mathcal E}}
\def\F{{\mathcal F}} \def\G{{\mathcal G}} \def\H{{\mathcal H}}

\def\J{{\mathcal J}}  
 \def\N{{\mathcal N}}
\def\O{{\mathcal O}} \def\P{{\mathcal P}}  
  \def\S{{\mathcal S}} 
\def\U{{\mathcal U}} \def\V{{\mathcal V}} 
\def\W{{\mathcal W}}

\def\r{{\mathbb R}}
\def\z{{\mathbb Z}}

\def\cuum{{\mathfrak c}}

\makeatother

\usepackage{babel}

\makeatother

\usepackage{babel}

\begin{document}

\title[core-compactness]{Core compactness and diagonality in spaces of open sets}

\author{Francis Jordan}

\author{Fr\'ed\'eric Mynard}

\address{}

\email{fmynard@georgiasouthern.edu}

\email{fejord@hotmail.com}

\maketitle
\global\long\global\long\def\union{\bigcup}

\global\long\global\long\def\then{\Longrightarrow}

\date{\today}
\begin{abstract}
We investigate when the space $\mathcal O_X$ of open subsets of a topological space $X$ endowed with the Scott topology is core compact. Such conditions turn out to be related to infraconsonance of $X$, which in turn is characterized in terms of coincidence of the Scott topology of $\mathcal O_X\times\mathcal O_X$ with the product of the Scott topologies of $\mathcal O_X$ at $(X,X)$. On the other hand, we characterize diagonality of $\mathcal O_X$ endowed with the Scott convergence and show that this space can be diagonal without being pretopological. New examples are provided to clarify the relationship between pretopologicity, topologicity and diagonality of this important convergence space.
\end{abstract}

\section{Introduction}

Definitions and notations concerning convergence structures follow
\cite{dol.initiation} and are gathered as an appendix at the end
of these notes. If $X$ is a topological space, we denote by $\O_{X}$
the set of its open subsets. Ordered by inclusion, it is a complete
lattice in which the \textit{Scott convergence} (in the sense of,
for instance, \cite{contlattices}) is given by \begin{equation}
U\in\lm\F\iff U\subseteq\union_{F\in\F}\intr\left(\bigcap_{O\in F}O\right),\label{eq:scottconvergence}\end{equation}
 where $\F$ is a filter on $\O_{X}$; and its topological modification,
the \textit{Scott topology}, has open sets composed of compact families
(%
\footnote{See the Appendix for definitions%
}). $\O_{X}$ can be identified with the set $C(X,\$)$ of continuous
functions from $X$ to the Sierpi\'{n}ski space $\$$ because the
indicator function of $A\subseteq X$ is continuous if and only if
$A$ is open. Via this identification, the convergence \eqref{eq:scottconvergence}
coincides with the continuous convergence $[X,\$]$ on $C(X,\$)$,
and its topological modification $T[X,\$]$ coincides with the Scott
topology.

On the other hand, for a general convergence space $X$, the underlying
set of $[X,\$]$ can still be identified with the collection $\O_{X}$
of open subsets of $X$ (or $TX$ ), but the characterization \eqref{eq:scottconvergence}
of convergence in $[X,\$]$ (when interpreted as a convergence on
$\O_{X}$) needs to be modified. Recall that a family $\S$ of subsets
of a convergence space $X$ is a \textit{cover} if every convergent
filter on $X$ contains an element of the family $\S$. In general,
$U\in\lm_{[X,\$]}\F$ if the family $\{\bigcap_{O\in F}O:F\in\F\}$
is a cover of $U$ (for the induced convergence).

Given a convergence space $X$, it is known (e.g., \cite{schwarz.compatible},
\cite{DM.mech}) that the following are equivalent: \begin{eqnarray}
\forall Y,\,\, T(X\times Y)\leq X\times TY;\label{eq:partcommute}\\
T(X\times[X,\$])\leq X\times T[X,\$];\\
\left[X,\$\right]=T[X,\$].\end{eqnarray}
 Let us call a convergence space $X$ satisfying this condition $T$-\textit{dual}.
In the case where $X$ is topological, the latter is well-known to
be equivalent to core compactness of $X$ (e.g., \cite{hofflawson},
\cite{schwarz.powers}). Recall that a topological space $X$ is \textit{core
compact} if for every $x$ and $O\in\O(x)$, there is $U\in\O(x)$
such that every open cover of $O$ has a finite subfamily that covers
$U$. In \cite{DM.mech}, a \textit{convergence space} is called \textit{core
compact} if whenever $x\in\lim\F$, there is $\G\leq\F$ with $x\in\lim\G$
and for every $G\in\G$ there is $G'\in\G$ such that $G'$ is compact
at $G$. A convergence space is called $T$-\textit{core compact}
if whenever $x\in\lim\F$ and $U\in\O_{TX}(x)$, there is $F\in\F$
that is compact at $U$. It is shown in \cite{DM.mech} that \begin{equation}
X\text{ is core compact }\then X\text{ is \ensuremath{T}-dual }\then X\text{ is \ensuremath{T}-core compact.}\label{eq:corevariants}\end{equation}
 The three notions clearly coincide if $X$ is topological. However,
so far, it was not known whether they do in general. At the end of
the paper, we provide examples (Example \ref{exa:tdualnotcorec} and
Example \ref{exa:TcorenotTdual}) showing that none of the arrows
can be reversed for general convergence spaces.

It was observed in \cite{dcpoconv} that if $X$ is topological, then
so is $[[X,\$],\$]$. Therefore $[X,\$]$ is then $T$-core compact,
which makes $[X,\$]$ for $X$ topological but not core compact a
natural candidate to distinguish core-compactness from $T$-core compactness.
This however fails, in view of Proposition \ref{pro:dualcore} below.
In the next section, we also investigate under what condition $T[X,\$]$
(that is $\O_{X}$ with the Scott topology) is core compact. This
question, while natural in itself, is motivated by its connection
with the (now recently solved) problem \cite[Problem 1.2]{dolmyn.isbell}
of finding a completely regular infraconsonant topological space that
is not consonant (see section 3 for definitions). We observe in Section
3 that $X$ is infraconsonant whenever $T[X,\$]$ is core compact
and we prove more generally that $X$ is infraconsonant if and only
if the Scott topology on $\O_{X}\times\O_{X}$ for the product order
coincides with the product of the Scott topologies at the point $(X,X)$
(Theorem \ref{th:prodscottinfraC}). Infraconsonance was introduced
while studying the Isbell topology on the set of real-valued continuous
functions over a topological space. In fact a completely regular space
$X$ is infraconsonant if and only if the Isbell topology on the set
of real-valued continuous functions on $X$ is a group topology \cite[Corollary 4.6]{DJM.group}.
On the other hand, the fact that the Scott topology on the product
does not coincide in general with the product of the Scott topologies
has been at the origin of a number of problems and errors (e.g., \cite[p.197]{contlattices}).
Therefore, Theorem \ref{th:prodscottinfraC} provides new motivations
to investigate infraconsonance.

In section 4, we show that, for a topological space $X$, despite
the fact that $[X,\$]$ is topological whenever it is pretopological,
$[X,\$]$ does not need to be diagonal in general. Diagonality of
$[X,\$]$ is characterized in terms of a variant of core-compactness
that do not need to coincide with core-compactness.

\section{core-compactness of $\O_{X}$}

As $[X,\$]$ can be identified with $\O_{X}$ for any convergence
space $X$, the space $[[X,\$],\$]$ has as underlying set the set
of Scott-open subsets of $\O_{X}$, that is, if $X$ is topological,
the set $\kappa(X)$ of openly isotone compact families on $X$. Note
that the family \[
\{U^{+}:=\{\A\in\kappa(X):U\in\A\}:U\in\O_{X}\}\]
 forms a subbase for a topology on $\kappa(X)$, called \textit{Stone
topology}. 

As observed in \cite[Proposition 5.2]{DM.transfer}, when $X$ is
topological, the convergence $[X,\$]$ is based in filters of the
form \begin{equation}
\O^{\natural}(\P):=\{\O(P):P\in\P\},\label{eq:basedual}\end{equation}
 where $\P$ is an \textit{ideal subbase} of open subsets of $X$,
that is, such that there is $P\in\P$ with $\bigcup_{Q\in\P_{0}}Q\subseteq P$
whenever $\P_{0}$ is a finite subfamily of $\P$. More precisely,
for every filter $\F$ on $[X,\$]$ with $U\in\lim_{[X,\$]}\F$ there
is an open cover $\P$ of $U$ that forms an ideal subbase, such that
$U\in\lm_{[X,\$]}\O^{\natural}(\P)$ and $\O^{\natural}(\P)\leq\F$.

Note also that \begin{equation}
\A\subseteq\B\then\A\in\lm_{[[X,\$],\$]}\{\B\}^{\uparrow},\label{eq:ptinbidual}\end{equation}
for every $\A$ and $\B$ in $\kappa(X)$. In particular if $O$ is
$[[X,\$],\$]$-open, $\A\in O$ and $\A\subseteq\B\in\kappa(X)$ then
$\B\in O$. \[
\]

\begin{prop}
\label{pro:dualcore} If $X$ is topological, then $[X,\$]$ is core
compact, so that $[[X,\$],\$]$ is topological. More specifically,
$[[X,\$],\$]$ can be identified with the space $\kappa(X)$ with
the Stone topology. \end{prop}
\begin{proof}
Let $U\in\lm_{[X,\$]}\O^{\natural}(\P)$ for an ideal subbase $\P$
of open subsets of $X$. Then for each $P\in\P$, the set $\O(P)$
is a compact subset of $[X,\$]$ because $P\in\lm_{[X,\$]}\O(P)$.
Indeed, $P=\intr\left(\bigcap_{O\in\O(P)}O\right)$.

$U^{+}$ is $[[X,\$],\$]$-open for each $U\in\O_{X}$. Indeed, if
$\A\in U^{+}\cap\lim_{[[X,\$],\$]}\F$ then $\left\{ \bigcap_{\B\in F}\B:F\in\F\right\} $
is a cover of $\A$ (in the sense of convergence) so that there is
$F\in\F$ with $\bigcap_{\B\in F}\B\in\{U\}^{\uparrow}$ because $U\in\lm_{[X,\$]}\{U\}^{\uparrow}\cap\A$.
In other words, $F\inc U^{+}$, so that $U^{+}\in\F$.

Conversely, if $O$ is $[[X,\$],\$]$-open and $\A\in O$, there is
$U\in\A$ such that $U^{+}\inc O$. Otherwise, for each $U\in\A$,
there is $\B\in\kappa(X)$ with $U\in\B$ and $\B\notin O$. In that
case, $\widehat{U}:=\{\B\in\kappa(X):U\in\B,\B\notin O\}\neq\emptyset$
for all $U\in\A$. Note also that in view of \ref{eq:ptinbidual},
$\B_{U}\cap\B_{V}\in\widehat{U}\cap\widehat{V}$ whenever $\B_{U}\in\widehat{U}$
and $\B_{V}\in\widehat{V}$. Therefore $\left\{ \bigcap_{i\in I}\widehat{U_{i}}:U_{i}\in\A:\card I<\infty\right\} $
is a filter-base generating a filter $\F$. This filter converges
to $\A$ in $[[X,\$],\$]$. To show that, we need to see that $\left\{ \bigcap_{\B\in\widehat{U}}\B:U\in\A\right\} $
is a cover of $\A$ for $[X,\$]$. In view of the form \ref{eq:basedual}
of a base for $[X,\$]$, it is enough to show that if $U_{0}\in\A$
and $\P$ is an ideal subbase of open subsets of $X$ covering $U_{0}$,
then there is $A\in\A$ with $\bigcap_{\B\in\widehat{A}}\B\in\O^{\natural}(\P)$.
Because $U_{0}\subseteq\bigcup_{P\in\P}P$ and $\A$ is a compact
family, there is a finite subfamily $\P_{0}$ of $\P$ such that $\bigcup_{P\in\P_{0}}P\in\A$.
Since $\P$ is an ideal subbase, there is $P\in\P\cap\A$. Then $\O(P)\subseteq\bigcap_{\B\in\widehat{P}}\B$,
which concludes the proof that $\A\in\lm_{[[X,\$],\$]}\F$. On the
other hand, $O\notin\F$, which contradicts the fact that $O$ is
open for $[[X,\$],\$]$.
\end{proof}
In order to investigate when $T[X,\$]$, that is, $\O_{X}$ with the
Scott topology, is core-compact, we will need notions and results
from \cite{DM.mech}. The concrete endofunctor $\Epi_{T}$ of the
category of convergence spaces (and continuous maps) is defined (on
objects) by \[
\Epi_{T}X=i^{-}[T[X,\$],\$]\]
 where $i:X\to[[X,\$],\$]$ is defined by $i(x)(f)=f(x)$. In view
of \cite[Theorem 3.1]{DM.mech} \begin{equation}
W\geq\Epi_{T}X\iff T[X,\$]\geq[W,\$],\label{eq:EpiT}\end{equation}
 where $X\geq W$ have the same underlying set. In particular, $X$
is $T$-dual if and only if $X\geq\Epi_{T}X$. A convergence space
$X$ is \textit{epitopological} if $i:X\to[[X,\$],\$]$ is initial
(in the category \textbf{Conv} of convergence spaces and continuous
maps). Epitopologies form a reflective subcategory \textbf{Epi} of
\textbf{Conv} and the (concrete) reflector is given (on objects) by
$\Epi X=i^{-}[[X,\$],\$]$. Because $[\Epi X,\$]=[X,\$]$, we assume
from now on that every space is epitopological. Observe that a topological
space is epitopological. Note that if $[X,\$]$ is $T$-dual, then
$\Epi X=X$ is topological. Therefore, in contrast to Proposition
\ref{pro:dualcore}, $[X,\$]$ is not $T$-dual if $X$ is not topological. 
\begin{prop}
\label{pro:notcore} Let $X$ be an epitopological space. Then $X$
is topological if and only if $[X,\$]$ is $T$-dual. 
\end{prop}
Note also that $\Epi X\leq\Epi_{T}X$ and that $\Epi_{T}\circ\Epi=\Epi_{T}$,
so that $\Epi_{T}$ restricts to an expansive endofunctor of \textbf{Epi}.
By iterating this functor, we obtain the coreflector on $T$-dual
epitopologies. More precisely, if $F$ is an expansive concrete endofunctor
of \textbf{C}, we define the transfinite sequence of functors $F^{\alpha}$
by $F^{1}=F$ and $F^{\alpha}X=F\left(\bigvee_{\beta<\alpha}F^{\beta}X\right)$.
For each epitopological space $X$, there is $\alpha(X)$ such that
\[
\Epi_{T}^{\alpha(X)}X=\Epi_{T}^{\alpha(X)+1}X:=d_{T}X.\]

\begin{prop}
\label{pro:Tdual} The class of $T$-dual epitopologies is concretely
coreflective in \textbf{Epi} and the coreflector is $d_{T}$. \end{prop}
\begin{proof}
The class of $T$-dual convergences is closed under infima because
\[
\left[\bigwedge_{i\in I}X_{i},Z\right]=\bigvee_{i\in I}[X_{i},Z].\]
 Indeed, if each $X_{i}$ is $T$-dual, then \[
\left[\bigwedge_{i\in I}X_{i},\$\right]=\bigvee_{i\in I}[X_{i},\$]=\bigvee_{i\in I}T[X_{i},\$]\leq T\left(\bigvee_{i\in I}[X_{i},\$]\right)=T\left(\left[\bigwedge_{i\in I}X_{i},\$\right]\right),\]
 and $\bigwedge_{i\in I}X_{i}$ is $T$-dual. The functor $\Epi_{T}$
is expansive on \textbf{Epi} and therefore, so is $d_{T}$. Moreover,
$d_{T}X$ is $T$-dual for each epitopological space $X$ because
\[
[d_{T}X,\$]=[\Epi_{T}^{\alpha(X)+1}X,\$]\leq T[\Epi_{T}^{\alpha(X)}X,\$]=T[d_{T}X,\$].\]
 Therefore, for each epitopological space $X$, there exists the coarsest
$T$-dual convergence $\overline{X}$ finer than $X$. By definition
$X\leq\overline{X}\leq d_{T}X$. Then $[\overline{X},\$]\leq[X,\$]$
and $[\overline{X},\$]$ is topological, so that $[\overline{X},\$]\leq T[X,\$]$.
But $\Epi_{T}X$ is the coarsest convergence with this property. Therefore
$\Epi_{T}X\leq\overline{X}=\Epi_{T}\overline{X}$ and $d_{T}X\leq\overline{X}$. \end{proof}
\begin{prop}
\label{pro:coretocore} If $X$ is a core compact topological space,
so is $T[X,\$]$. \end{prop}
\begin{proof}
Under these assumptions, $[X,\$]=T[X,\$]$ and $[X,\$]$ is $T$-dual
by Proposition \ref{pro:dualcore}. Therefore $T[X,\$]$ is a core
compact topology. 
\end{proof}
However, if $X$ is a non-topological $T$-dual convergence space
(%
\footnote{Such convergences exist: take for a instance a non-locally compact
Hausdorff regular topological $k$-space. Then $X=TK_{h}X$ but $X<K_{h}X$
so that $K_{h}X$ is non-topological.%
}), then $[X,\$]=T[X,\$]$ is \textit{not} core compact, by Proposition
\ref{pro:notcore}. In other words, we have: 
\begin{prop}
If $X$ is $T$-dual then $X$ is topological if and only if $[X,\$]=T[X,\$]$
is core compact. 
\end{prop}
In particular, $d_{T}X$ is topological if and only if $[d_{T}X,\$]$
is core compact. 
\begin{thm}
\label{th:quotientofcore} If $X\geq Td_{T}X$ then $T[X,\$]$ is
core compact if and only if $X$ is a core compact topological space. \end{thm}
\begin{proof}
We already know that if $X$ is a core compact topological space then
$[X,\$]=T[X,\$]$ and that $[X,\$]$ is core compact by Proposition
\ref{pro:dualcore}. Conversely, if $T[X,\$]$ is core compact then
$[T[X,\$],\$]$ is topological, so that $\Epi_{T}X$ is topological.
Under our assumptions, we have \[
X\geq Td_{T}X\geq T\Epi_{T}X=\Epi_{T}X,\]
 so that $X$ is $T$-dual. Therefore $[X,\$]=T[X,\$]$ is core compact
and, in view of Proposition \ref{pro:notcore}, $X$ is topological,
and $T$-dual, hence a core compact topological space.
\end{proof}
Note that, at least among Hausdorff topological spaces, Theorem \ref{th:quotientofcore}
generalizes \cite[Corollary 3.6]{corecompactdual} that states that
if $X$ is first countable, then $X$ is core compact if and only
if $T[X,\$]$ is core compact. Indeed, the locally compact coreflection
$KX$ of a Hausdorff topological space is $T$-dual so that $d_{T}X\leq KX$.
Hence if $X$ is a Hausdorff topological $k$-space, that is $X=TKX$,
(in particular a first-countable space) then $X\geq Td_{T}X$. We
will see in the next section that similarly, if $X$ is a consonant
topological space, then $T[X,\$]$ is core compact if and only if
$X$ is locally compact. 
\begin{problem}
\label{prob:noncorewithcoredual} Are there completely regular non
locally compact topological spaces $X$ such that $T[X,\$]$ is core
compact? 
\end{problem}
Of course, in view of the observations above, such a space cannot
be a $k$-space or consonant.

\section{Core compact dual, Consonance, and infraconsonance}

A topological space is \textit{consonant} \cite{DGL.kur} if every
Scott open subset $\A$ of $\O_{X}$ is compactly generated, that
is, there are compact subsets $(K_{i})_{i\in I}$ of $X$ such that
$\A=\union_{i\in I}\O(K_{i})$. A space is \textit{infraconsonant}
\cite{dolmyn.isbell} if for every Scott open subset $\A$ of $\O_{X}$
there is a Scott open set $\C$ such that $\C\vee\C\subseteq\A$,
where $\C\vee\C:=\{C\cap D:C,D\in\C\}$. The notion's importance stems
from Theorem \ref{th:infraequiv} below. If the set $C(X,Y)$ of continuous
functions from $X$ to $Y$ is equipped with the Isbell topology (%
\footnote{whose sub-basic open sets are given by \[
[\A,U]:=\left\{ f\in C(X,Y):\exists A\in\A,\: f(A)\subseteq U\right\} ,\]
where $\A$ ranges over openly isotone compact families on $X$ and
$U$ ranges over open subsets of $Y$.%
}), we denote it $C_{\kappa}(X,Y)$, while $C_{k}(X,Y)$ denotes $C(X,Y)$
endowed with the compact-open topology. Note that $C_{\kappa}(X,\$)=T[X,\$]$. 
\begin{thm}
\label{th:infraequiv} \cite{DJM.group} Let $X$ be a completely
regular topological space. The following are equivalent: 
\begin{enumerate}
\item $X$ is infraconsonant; 
\item addition is jointly continuous at the zero function in $C_{\kappa}(X,\r)$; 
\item $C_{\kappa}(X,\r)$ is a topological vector space; 
\item $\cap:T[X,\$]\times T[X,\$]\to T[X,\$]$ is jointly continuous. 
\end{enumerate}
\end{thm}
On the other hand, if $X$ is consonant then $C_{\kappa}(X,\r)=C_{k}(X,\r)$
so that consonance provides an obvious sufficient condition for $C_{\kappa}(X,\r)$
to be a topological vector space. Hence Theorem \ref{th:infraequiv}
becomes truly interesting if completely regular examples of infraconsonant
non consonant spaces can be provided \cite[Problem 1.2]{dolmyn.isbell}.
The first author recently obtained the first example of this kind.
The following results show that a space answering positively Problem
\ref{prob:noncorewithcoredual} would necessarily be infraconsonant
and non-consonant and might provide an avenue to construct new examples. 
\begin{thm}
If $X$ is topological and $T[X,\$]$ is core compact then $X$ is
infraconsonant. \end{thm}
\begin{proof}
\cite[Lemma 3.3]{dolmyn.isbell} shows the equivalence between (1)
and (4) in Theorem \ref{th:infraequiv}, and that the implication
(4)$\then$(1) does not require any separation. Therefore, it is enough
to show that $\cap:T[X,\$]\times T[X,\$]\to T[X,\$]$ is continuous.
Since $X$ is topological, $[X,\$]$ is $T$-dual by Proposition \ref{pro:dualcore}.
In view of \eqref{eq:partcommute} \[
T([X,\$]\times[X,\$])\leq[X,\$]\times T[X,\$]\]
 so that $T([X,\$]\times[X,\$])\leq T([X,\$]\times T[X,\$])$. If
$T[X,\$]$ is core compact, hence $T$-dual then $T([X,\$]\times T[X,\$])\leq T[X,\$]\times T[X,\$]$
so that \begin{equation}
T([X,\$]\times[X,\$])\leq T[X,\$]\times T[X,\$].\label{eq:Tfullcomm}\end{equation}
 Therefore the continuity of $\cap:[X,\$]\times[X,\$]\to[X,\$]$ implies
that of $\cap:T([X,\$]\times[X,\$])\to T[X,\$]$ because $T$ is a
functor, and in view of \eqref{eq:Tfullcomm}, that of $\cap:T[X,\$]\times T[X,\$]\to T[X,\$]$. \end{proof}
\begin{thm}
\label{th:Okcore} Let $X$ be a topological space. If $C_{k}(X,\$)$
is core compact then $X$ is locally compact. \end{thm}
\begin{proof}
If $X$ is not locally compact, then $C_{k}(X,\$)\ngeq[X,\$]$ (e.g.,
\cite[2.19]{schwarz.powers}) so that there is $U_{0}\in\O_{X}$ with
$U_{0}\notin\lm_{[X,\$]}\N_{k}(U_{0})$. Therefore, there is $x_{0}\in U_{0}$
such that $x_{0}\notin\intr\left(\bigcap_{V\in\O(K)}V\right)$ whenever
$K$ is a compact subset of $X$ with $K\subseteq U_{0}$. In other
words, for each such $K$ and for each $U\in\O(x_{0})$ there is $V_{U}\in\O(K)$
and $x_{U}\in U\setminus V_{U}$. Then $C_{k}(X,\$)$ is not core
compact at $U_{0}$. Indeed, there is $U_{0}\in\O(x_{0})$ such that
for every compact set $K$ with $K\subseteq U_{0}$, the $k$-open
set $\O(K)$ is not relatively compact in $\O(x_{0})$. To see that,
consider the cover $\S:=\{\O(x_{U}):U\in\O(x_{0})\}$ of $\O(x_{0})$.
No finite subfamily of $\S$ covers $\O(K)$ because for any finite
choice of $U_{1},\dots,U_{n}$ in $\O(x_{0})$, we have $W:=\cap_{i=1}^{i=n}V_{U_{i}}\in\O(K)$
but $W\notin\cup_{i=1}^{i=n}\O(x_{U_{i}})$. 
\end{proof}
Note that a Hausdorff topological space $X$ is locally compact if
and only if it is core compact, and that the Scott open filter topology
on $\O(X)$ then coincides with $C_{k}(X,\$)$ (e.g., \cite[Lemma II.1.19]{contlattices}).
Hence Theorem \ref{th:Okcore} could also be deduced (for the Hausdorff
case) from \cite[Corollary 3.6]{corecompactdual}. 
\begin{cor}
If $X$ is a consonant topological space such that $T[X,\$]$ is core
compact, then $X$ is locally compact. 
\end{cor}

\section{Scott topology of the product versus product of Scott topologies}

We now turn to a new characterization of infraconsonance, which motivates
further the systematic investigation of the notion. 
\begin{prop}
\label{pro:Scottprod} $T([X,\$]^{2})$ is the Scott topology on $\O_{X}\times\O_{X}$. \end{prop}
\begin{thm}
\label{th:prodscottinfraC} A space $X$ is infraconsonant if and
only if the product $(T[X,\$])^{2}$ of the Scott topologies on $\O_{X}$
and the Scott topology $T([X,\$]^{2})$ on the product $\O_{X}\times\O_{X}$
coincide at $(X,X)$. \end{thm}
\begin{lem}
\label{lem:coordcompact} A subset $\S$ of $\O_{X}\times\O_{X}$
is $[X,\$]^{2}$-open if and only if 
\begin{enumerate}
\item $\S=\S^{\uparrow}$, that is, if $(U,V)\in\S$ and $U\subseteq U'$,
$V\subseteq V'$ then $(U',V')\in\S$; 
\item $\S$ is coordinatewise compact, that is, \[
\left(\union_{i\in I}O_{i},\union_{j\in J}V_{j}\right)\in\S\then\exists I_{0}\in[I]^{<\omega},J_{0}\in[J]^{<\omega}:\left(\union_{i\in I_{0}}O_{i},\union_{j\in J_{0}}V_{j}\right)\in\S\]
 
\end{enumerate}
\end{lem}
\begin{proof}
Assume $\S$ is $[X,\$]^{2}$-open and let $(U,V)\in\S$ and $U\subseteq U'$,
$V\subseteq V'$. Then $(U,V)\in\lm_{[X,\$]^{2}}\{(U',V')\}^{\uparrow}$
so that $(U',V')\in\S$. Assume now that $\left(\union_{i\in I}O_{i},\union_{j\in J}V_{j}\right)\in\S$.
Then $\{\O\left(\union_{i\in F}O_{i}\right):F\in[I]^{<\infty}\}$
is a filter-base for a filter $\gamma$ on $\O_{X}$ such that $\union_{i\in I}O_{i}\in\lm_{[X,\$]}\gamma$
and $\{\O\left(\union_{j\in D}V_{j}\right):D\in[J]^{<\infty}\}$ is
a filter-base for a filter $\eta$ on $\O_{X}$ such that $\union_{j\in J}V_{j}\in\lm_{[X,\$]}\eta$.
Hence $\S\in\gamma\times\eta$ because $\S$ is $[X,\$]^{2}$-open.
Therefore, there are finite subsets $I_{0}$ of $I$ and $J_{0}$
of $J$ such that $\O\left(\union_{i\in I_{0}}O_{i}\right)\times\O\left(\union_{j\in J_{0}}V_{j}\right)\subseteq\S$,
so that $\left(\union_{i\in I_{0}}O_{i},\union_{j\in J_{0}}V_{j}\right)\in\S$.

Conversely, assume that $\S$ satisfies the two conditions of the
Lemma and $(U,V)\in\S\cap\lm_{[X,\$]^{2}}(\gamma\times\eta)$. Since
$U\subseteq\union_{\G\in\gamma}\intr\left(\bigcap_{G\in\G}G\right)$
and $V\subseteq\union_{\H\in\eta}\intr\left(\bigcap_{H\in\H}H\right)$,
we have, by the first condition, that $\left(\union_{\G\in\gamma}\intr\left(\bigcap_{G\in\G}G\right),\union_{\H\in\eta}\intr\left(\bigcap_{H\in\H}H\right)\right)\in\S$.
By the second condition, there are $\G_{1},\ldots,\G_{k}\in\gamma$
and $\H_{1},\ldots,\H_{n}\in\eta$ such that \[
\left(\union_{i=1}^{k}\intr\left(\bigcap_{G\in\G_{i}}G\right),\union_{j=1}^{n}\intr\left(\bigcap_{H\in\H_{j}}H\right)\right)\in\S.\]
 Therefore $\left(\intr\left(\bigcap_{G\in\bigcap_{i=1}^{k}\G_{i}}G\right),\intr\left(\bigcap_{H\in\bigcap_{j=1}^{n}\H_{j}}H\right)\right)\in\S$
so that \[
\left(\bigcap_{i=1}^{k}\G_{i},\bigcap_{j=1}^{n}\H_{j}\right)\subseteq\S,\]
 and $\S\in\gamma\times\eta$. 
\end{proof}

\begin{proof}[Proof of Proposition \ref{pro:Scottprod}]
 In view of Lemma \ref{lem:coordcompact}, every $[X,\$]^{2}$-open
subset of $\O_{X}\times\O_{X}$ is Scott open. Conversely, consider
a Scott open subset $\S$ of $\O_{X}\times\O_{X}$. We only have to
check that $\S$ statisfies the second condition in Lemma \ref{lem:coordcompact}.
Let $(\union_{i\in I}O_{i},\union_{j\in J}V_{j})\in\S$. The set $D:=\lbrace\left(\union_{i\in I_{0}}O_{i},\union_{j\in J_{0}}V_{j}\right):I_{0}\in[I]^{<\omega},J_{0}\in[J]^{<\omega}\rbrace$
is a directed subset of $\O_{X}\times\O_{X}$ (for the coordinatewise
inclusion order) whose supremum is $(\union_{i\in I}O_{i},\union_{j\in J}V_{j})$.
As $\S$ is Scott-open, there are finite subsets $I_{0}$ of $I$
and $J_{0}$ of $J$ such that $\left(\union_{i\in I_{0}}O_{i},\union_{j\in J_{0}}V_{j}\right)\in\S$. \end{proof}
\begin{lem}
\label{lem:inter} If $\A\in\kappa(X)$ then $\S_{\A}:=\{(U,V)\in\O_{X}\times\O_{X}:U\cap V\in\A\}^{\uparrow}$
is $[X,\$]^{2}$-open. \end{lem}
\begin{proof}
Let $(\union_{i\in I}O_{i},\union_{j\in J}V_{j})\in\S_{\A}$. Then
\[
(\union_{i\in I}O_{i})\cap(\union_{j\in J}V_{j})=\union_{(i,j)\in I\times J}O_{i}\cap V_{j}\in\A.\]
 By compactness of $\A$, there is a finite subset $I_{0}$ of $I$
and a finite subset $J_{0}$ of $J$ such that $\union_{(i,j)\in I_{0}\times J_{0}}O_{i}\cap V_{j}\in\A$,
so that $(\union_{i\in I_{0}}O_{i},\union_{j\in J_{0}}V_{j})\in\S_{\A}$.
In view of Lemma \ref{lem:coordcompact}, $\S_{\A}$ is $[X,\$]^{2}$-open. \end{proof}
\begin{lem}
\label{lem:down} If $\S$ is $[X,\$]^{2}$-open, then \[
\downarrow\S:=\O_{X}(\{U\cup V:(U,V)\in\S\})\]
 is a compact family on $X$. \end{lem}
\begin{proof}
If $U\cup V\subseteq\union_{i\in I}O_{i}$ for some $(U,V)\in\S$
then $\left(\union_{i\in I}O_{i},\union_{i\in I}O_{i}\right)\in\S$
so that, in view of Lemma \ref{lem:coordcompact}, there is a finite
subset $I_{0}$ of $I$ such that $\left(\union_{i\in I_{0}}O_{i},\union_{i\in I_{0}}O_{i}\right)\in\S$.
Hence $\union_{i\in I_{0}}O_{i}\in\downarrow\S$. 
\end{proof}

\begin{proof}[Proof of Theorem \ref{th:prodscottinfraC}]
 Suppose that $X$ is infraconsonant. Note that $(T[X,\$])^{2}\leq T([X,\$]^{2})$
is always true, so that we only have to prove the reverse inequality
at $(X,X)$. Consider an $[X,\$]^{2}$-open neighborhood $\S$ of
$(X,X)$. By Lemma \ref{lem:down}, the family $\downarrow\S$ is
compact. By infraconsonance, there is $\C\in\kappa(X)$ with $\C\vee\C\subseteq\downarrow\S$.
Note that \[
\C\times\C\subseteq\S,\]
 because if $(C_{1},C_{2})\in\C\times\C$ then $C_{1}\cap C_{2}\in\downarrow\S$
so that $C_{1}\cap C_{2}\supseteq U\cup V$ for some $(U,V)\in\S$,
and therefore $(C_{1},C_{2})\in\S$.

Conversely, assume that $\N_{[X,\$]^{2}}(X,X)=\N_{T[X,\$]^{2}}(X,X)$
and let $\A\in\kappa(X)$. By Lemma \ref{lem:inter}, $\S_{\A}\in\N_{[X,\$]^{2}}(X,X)$
so that $\S_{\A}\in\N_{T[X,\$]^{2}}(X,X)$. In other words, there
are families $\B$ and $\C$ in $\kappa(X)$ such that $\B\times\C\inc\S_{\A}$.
In particular $\D:=\B\cap\C$ belongs to $\kappa(X)$ and satisfies
$\D\times\D\inc\S_{\A}$. By definition of $\S_{\A}$, we have that
$\D\vee\D\inc\A$ and $X$ is infraconsonant. 
\end{proof}

\section{Topologicity, pretopologicity and diagonality of $[X,\$]$}

A convergence space $X$ is \textit{diagonal} if for every selection
$\S[\cdot]:X\to\mathbb{F}X$ with $x\in\lm_{X}\S[x]$ for all $x\in X$
and every filter $\F$ with $x_{0}\in\lm_{X}\F$ the filter \begin{equation}
\S[\F]:=\union_{F\in\F}\bigcap_{x\in F}\S[x]\label{eq:contour}\end{equation}
 converges to $x_{0}$. If this property only holds when $\F$ is
additionally principal, we say that $X$ is $\mathbb{F}_{1}$-\emph{diagonal.
}Of course, every topology is diagonal. In fact a convergence is topological
if and only if it is both pretopological and diagonal (e.g., \cite{DG.pretop}). 

In order to compare our condition for diagonality of $[X,\$]$ with
core-compactness, we first rephrase the condition of core-compactness. 
\begin{lem}
\label{lem:corecompfamily}A topological space is core compact if
and only if for every $x\in X$, every $U\in\O(x)$ and every family
$\mathbb{H}$ of filters on $X$, we have\begin{equation}
\forall\H\in\mathbb{H}:\adh\H\cap U=\emptyset\then x\notin\adh\bigwedge_{\H\in\mathbb{H}}\H.\label{eq:corefamily}\end{equation}
\end{lem}
\begin{proof}
If $X$ is core compact, then there is $V\in\O(x)$ which is relatively
compact in $U$. If $\adh\H\cap U=\emptyset$ for every $\H\in\mathbb{H}$,
then $U\subseteq\bigcup_{H\in\H}(\cl H)^{c}$ so that, by relative
compactness of $V$ in $U$ there is, for each $\H\in\mathbb{H},$
a set $H_{\H}\in\H$ with $V\cap\cl H_{\H}=\emptyset$. Then $\bigcup_{\H\in\mathbb{H}}H_{\H}\in\bigwedge_{\H\in\mathbb{H}}\H$
but $\bigcup_{\H\in\mathbb{H}}H_{\H}\cap V=\emptyset$ so that $x\notin\adh\bigwedge_{\H\in\mathbb{H}}\H$.

Conversely, if (\ref{eq:corefamily}) is true, consider the family
$\mathbb{H}:=\{\H\in\mathbb{F}X:\adh\H\cap U=\emptyset\}$. In view
of (\ref{eq:corefamily}), $x\notin\adh\bigwedge_{\H\in\mathbb{H}}\H$
so that there is $V\in\O(x)$ such that $V\notin\left(\bigwedge_{\H\in\mathbb{H}}\H\right)^{\#}$.
Now $V$ is relatively compact in $U$ because any filter than meshes
with $V$ cannot be in $\mathbb{H}$ and has therefore adherence point
in $U$. 
\end{proof}
Recall that $[X,\$]=P[X,\$]$ if and only if $X$ is $T$-core compact,
and that, if $X$ is topological, $[X,\$]$ is topological whenever
it is pretopological. While the latest follows for instance from the
results of \cite{DM.mech}, it seems difficult to find an elementary
argument in the literature, which is why we include the following
proposition, which also illustrates the usefulness of Lemma \eqref{lem:corecompfamily}.
\begin{prop}
\label{prop:PtoT} If $X$ is topological and $[X,\$]$ is pretopological,
then $[X,\$]$ is topological. \end{prop}
\begin{proof}
We will show that under these assumptions, $X$ satisfies \eqref{eq:corefamily}.
Let $x\in X$ and $U\in\O(x)$. Let $\mathbb{H}$ be a family of filters
satisfying the hypothesis of \eqref{eq:corefamily}. Let $\H\in{\mathbb{H}}$.
Consider the filter base $\H^{*}:=\{\O(X\setminus\cl(H))\colon H\in\H\}$
on $[X,\$]$. Since $\adh(\H)\cap U=\emptyset$, it follows that $U\in\lim\H^{*}$.
Since $[X,\$]$ is pretopological, $U\in\lim\bigwedge_{\H\in{\mathbb{H}}}\H^{*}$.
In particular, there exist, for each $\H\in{\mathbb{H}}$, a $H_{\H}\in\H$
such that\begin{eqnarray*}
x\in\intr\left(\bigcap\bigcup_{\H\in{\mathbb{H}}}\O(X\setminus\cl(H_{\H}))\right) & = & \intr\left(\bigcap_{\H\in{\mathbb{H}}}(X\setminus\cl(H_{\H}))\right)\\
 & = & \intr\left(X\setminus(\bigcup_{\H\in{\mathbb{H}}}\cl(H_{\H}))\right)\\
 & \subseteq & X\setminus\cl(\bigcup_{\H\in{\mathbb{H}}}H_{\H}).\end{eqnarray*}
 Thus, $x\notin\adh(\bigwedge_{\H\in{\mathbb{H}}}\H)$.
\end{proof}
In other words, if $[X,\$]$ is pretopological it is also diagonal,
provided that $X$ is topological. We will see that even if $X$ is
topological, $[X,\$]$ is not always diagonal. Moreover it can be
diagonal without being pretopological (examples \ref{exa:diagnonP*}
and \ref{exa:diagnonP}) and pretopological but not diagonal (Example
\ref{exa:TcorenotTdual}), when $X$ is no longer assumed to be topological.

We call a topological space \emph{injectively core compact} if for
every $x\in X$ and $U\in\O(x)$ the conclusion of (\ref{eq:corefamily})
holds for every family $\mathbb{H}$ of filters such that there is
an injection $\theta:\mathbb{H}\to\O(U)$ satisfying $\adh\H\cap\theta(\H)=\emptyset$
for each $\H\in\mathbb{H}$. As such a family $\mathbb{H}$ clearly
satisfies the premise of (\ref{eq:corefamily}), every core compact
space is in particular injectively core compact. 
\begin{thm}
\label{thm:diagonaldual}Let $X$ be a topological space. The following
are equivalent: 
\begin{enumerate}
\item $X$ is injectively core compact; 
\item $[X,\$]$ is diagonal; 
\item $[X,\$]$ is $\mathbb{F}_{1}$-diagonal. 
\end{enumerate}
\end{thm}
\begin{proof}
(1)$\then$(2): Let $\S[\centerdot]:\O_{X}\to\mathbb{F}\O_{X}$ be
a selection for $[X,\$]$ and let $U\in\lm_{[X,\$]}\F$. If $x\in U$,
there is $F\in\F$ such that $x\in\intr\left(\bigcap_{O\in F}O\right):=V$.
Note that $F\subseteq\O(V)$. For each $O\in F$, consider the filter
$\H_{O}$ on $X$ generated by $\{\cl_{X}$$\left(\bigcup_{W\in S}W^{c}\right):S\in\S[O]\}$.
Because $O\in\lm_{[X,\$]}\S[O]$, we have that $\adh_{X}\H_{O}\cap O=\emptyset$.
Because $X$ is injectively core compact and $\mathbb{H}:=\{\H_{O}:O\in F\}$
satisfies the required condition (with $\theta(\H_{O})=O$ ), we conclude
that $x\notin\adh_{X}\bigwedge_{O\in F}\H_{O}$. In other words, there
is $H\in\bigwedge_{O\in F}\H_{O}$ such that $x\notin\cl_{X}H$, that
is, $x\in\intr_{X}H^{c}$. Therefore, for each $O\in F$ there is
$S_{O}\in\S[O]$ such that \[
x\in\intr\left(\bigcap_{O\in F}\intr\left(\bigcap_{W\in S_{O}}W\right)\right)\subseteq\intr\left(\bigcap_{W\in\bigcup_{O\in F}S_{O}}W\right).\]
 In other words, there is $F\in\F$ and $M\in\bigwedge_{O\in F}\S[O]$
such that $x\in\intr_{X}\left(\bigcap_{W\in M}W\right)$, that is,
$U\in\lim_{[X,\$]}\S[\F]$.

(2)$\then$(3) is clear. (3)$\then$(1): Suppose $X$ is not injectively
core compact. Then there is $x\in X$, $U\in\O(x)$ and a family $\mathbb{H}$
of filters on $X$ with an injective map $\theta\colon\mathbb{H}\to\O(U)$
such that $\theta(\H)\cap\adh_{X}\H=\emptyset$ for each $\H\in\mathbb{H}$
but $x\in\adh_{X}\bigwedge_{\H\in\mathbb{H}}\H$. Define a relation
$\sim$ on ${\mathbb{H}}$ by $H_{1}\sim H_{2}$ provided that the
collections $\{\cl(H)\colon H\in\H_{1}\}$ and $\{\cl(H)\colon H\in\H_{2}\}$
both generate the same filter. Clearly, $\sim$ is an equivalance
relation. Let ${\mathbb{H}}^{*}\subseteq{\mathbb{H}}$ be such that
${\mathbb{H}}^{*}$ contains exactly one element of each equivalance
class of $\sim$. For each $\H\in{\mathbb{H}}^{*}$ let $\H^{*}$
be the filter with base $\{\cl(H)\colon H\in\H\}$. Let ${\mathbb{J}}=\{\H^{*}\colon\H\in{\mathbb{H}}^{*}\}$.

Define $\theta^{*}\colon{\mathbb{J}}\to\O(U)$ so that $\theta^{*}(\J)=\theta(\H)$,
where $\H\in{\mathbb{H}}^{*}$ is such that $\J=\H^{*}$. It is easily
checked that $\theta^{*}$ is injective. Since $\adh(\H^{*})=\adh(\H)$
for every $\H\in{\mathbb{H}}^{*}$, we have $\theta^{*}(\J)\cap\adh(\J)=\emptyset$.
It is also easy to check that $x\in\adh\left(\bigwedge_{\J\in{\mathbb{J}}}\J\right)$.

For each $\J\in{\mathbb{J}}$, the filter $\widetilde{\J}$ generated
on $\O_{X}$ by the filter-base $\left\{ O_{X}(X\setminus J)\colon J\in\J\right\} $
converges to $\theta^{*}(\J)$. Consider now the subset $\theta^{*}(\mathbb{J})$
of $\O(U)\subseteq\O_{X}$ and the selection $\S[\centerdot]:\O_{X}\to\mathbb{F}\O_{X}$
defined by $\S[\theta(\J)]=\widetilde{\J}$ for each $\J\in\mathbb{J}$
and $\S[O]=\{O\}^{\uparrow}$ for $O\notin\theta^{*}(\mathbb{J})$.
This is indeed a well-defined selection because $\theta^{*}$ is injective.

Notice that $U\in\lim_{[X,\$]}\theta^{*}(\mathbb{J})$ because $\theta^{*}(\mathbb{J})\subseteq\O(U)$.
Let $L\in\S[\theta^{*}(\mathbb{J})]$. We may pick from each $\J\in{\mathbb{J}}$
a closed set $J_{\J}\in\J$ such that $\bigcup_{\J\in{\mathbb{J}}}\O_{x}(X\setminus J_{\J})\subseteq L$.
Let $V$ be an open neighborhood of $x$. Since $x\in\adh_{X}\bigwedge_{\J\in\mathbb{J}}\J$
and $\bigcup_{\J\in{\mathbb{J}}}J_{\J}\in\bigwedge_{\J\in\mathbb{J}}\J$,
there is an $\J_{0}\in{\mathbb{J}}$ such that $V\cap J_{\J_{0}}\neq\emptyset$.
Since $V\not\subseteq X\setminus J_{\J_{0}}$ and $X\setminus J_{\J_{0}}\in\O_{X}(X\setminus J_{J_{0}})\subseteq L$,
$V\not\subseteq\bigcap\O_{X}(X\setminus J_{\J_{0}})$. Since $\O_{X}(X\setminus J_{\J_{0}})\subseteq L$,
$V\not\subseteq\bigcap L$. Since $V$ was an arbitrary neighborhood
of $x$, $x\not\in\intr(\bigcap L)$. Thus, $U\notin\S[\theta^{*}(\mathbb{J})]$.
Therefore, $[X,\$]$ is not $\mathbb{F}_{1}$-diagonal. 
\end{proof}
A cardinal number $\kappa$ is \emph{regular }if a union of less than
$\kappa$-many sets of cardinality less than $\kappa$ has cardinality
less than $\kappa$. A strong limit cardinal $\kappa$ is a cardinal
for which $\card(2^{A})<\kappa$ whenever $\card(A)<\kappa$. A \emph{strongly
inaccessible cardinal }is a regular strong limit cardinal. Uncountable
strongly inaccessible cardinals cannot be proved to exist within ZFC,
though their existence is not known to be inconsistent with ZFC. Let
us denote by ({*}) the assumption that such a cardinal exist. 
\begin{example}[{A Hausdorff space $X$ such that $[X,\$]$ is diagonal but not pretopological
under ({*})}]
\label{exa:diagnonP*} Assume that $\kappa$ is a (uncountable) strong
limit cardinal. Let $X$ be the subspace of $\kappa\cup\left\{ \kappa\right\} $
endowed with the order topology, obtained by removing all the limit
ordinals but $\kappa$. Since $X$ is a non locally compact Hausdorff
topological space, $[X,\$]$ is not pretopological. To show that $X$
is injectively core compact, we only need to consider $x=\kappa$
and $U\in\O(\kappa)$ in the definition, because $\kappa$ is the
only non-isolated point of $X$. Let $\mathbb{H}$ be a family of
filters on $X$ admitting an injective map $\theta:\mathbb{H}\to\O(U)$
such that $\adh\H\cap\theta(\H)=\emptyset$ for each $\H\in\mathbb{H}.$
Let $\beta$ be the least element of $U$. For each $\H\in\mathbb{H}$
there is $H_{\H}\in\H$ such that $\beta\notin H_{\H}$ so that $\card(H_{})<\beta$.
Moreover, $\card\mathbb{H}\leq\card\O(U)=2^{\beta}.$ Since $\kappa$
is a strong limit cardinal, $\card\mathbb{H}<\kappa$. Since $\kappa$
is regular, $\bigcup_{\H\in\mathbb{H}}H_{\H}<\kappa$ so that $\kappa\notin\adh\bigwedge_{\H\in\mathbb{H}}\H$. 
\end{example}
We do not know if the existence of large cardinals is necessary for
the construction of a Hausdorff space $X$ such that $[X,\$]$ is
diagonal and not pretopological. However, we can construct in ZFC
a $T_{0}$ space $X$ such that $[X,\$]$ is diagonal and not pretopological.
\begin{example}[{A $T_{0}$ space $X$ such that $[X,\$]$ is diagonal but not pretopological.}]
\label{exa:diagnonP} Let $\z$ stand for integers, $\cuum$ be the
cardinality of the continuum, and $\cuum^{+}$ be the cardinal successor
of $\cuum$. Let $\infty$ be a point that is not in $\cuum^{+}\times\z$
and $X=\{\infty\}\cup(\cuum^{+}\times\z)$. For each $(\alpha,n)\in\cuum^{+}\times\mathbb{Z}$
define $S_{\alpha,n}=\{(\beta,k)\colon\alpha\leq\beta\text{ and }n\leq k\}$.
For each $\alpha\in\cuum^{+}$, let $T_{\alpha}=\{(\beta,k)\colon\alpha\leq\beta\text{ and }k\in\z\}\cup\{\infty\}$.
Topologize $X$ by declaring all sets of the form $T_{\alpha}$ and
$S_{\alpha,n}$ to be sub-basic open sets.

We show that $X$ is not core compact at $\infty$. Let $U$ be a
neighborhood of $\infty$. There is an $\alpha$ such that $T_{\alpha}\subseteq U$.
Notice that $T_{\alpha+1}\cup\{S_{0,n}\colon n\in\z\}$ is a cover
of $X$ but no finite subcollection covers $T_{\alpha}$. Thus, $X$
is not core compact at $\infty$. In particular, $[X,\$]$ is not
pretopological.

Let $(\alpha,n)\in X\setminus\{\infty\}$. Let $U$ be an open neighborhood
of $(\alpha,n)$. Since $(\alpha,n)\in U$ it follows from the way
we chose our sub-base that $T_{\alpha,n}\subseteq U$. Since $(\alpha,n)$
has a minimal open neighborhood, $X$ is core compact at $(\alpha,n)$.

Let $V$ be an open neighborhood of $\infty$. There is an $\alpha$
such that $T_{\alpha}\subseteq V$. Let $U\subseteq X$ be an open
superset of $V$. For every $n\in\z$ $U\cap(\cuum^{+}\times\{n\})\neq\emptyset$.
For each $n\in\z$ define $\alpha_{n}=\min\{\beta\colon(\beta,n)\in U$.
Notice that $\{\beta\colon\alpha_{n}\leq\beta\}\times\{n\}=U\cap(\cuum^{+}\times\{n\})$
and $\alpha_{n}\leq\alpha$. Since each open superset of $V$ will
determine a unique sequence $(\alpha_{n})_{n\in z}$, it follows that
the open supersets of $V$ can injectively be mapped into the countable
sequences on $\{\beta\colon\beta\leq\alpha\}\times\z$. Since $\{\beta\colon\beta\leq\alpha\}\times\z$
has cardinality at most $\cuum$, $\{\beta\colon\beta\leq\alpha\}\times\z$
has at most $\cuum$-many countable sequences. Thus, $V$ has most
$\cuum$-many supersets.

Let $V$ be an open neighborhood of $\infty$, ${\mathbb{H}}$ be
a collection of filters, and $\theta\colon{\mathbb{H}}\to\O_{X}(V)$
be an injection such that $\adh(\H)\cap\theta(\H)=\emptyset$ for
every $\H\in{\mathbb{H}}$. Since $V$ has at most $\cuum$-many open
supersets, $|{\mathbb{H}}|\leq\cuum$. Let $\H\in{\mathbb{H}}$. Since
$\infty\notin\adh\H$, there is an $\alpha_{\H}\in\cuum^{+}$ such
that $\adh(\H)\cap T_{\alpha_{\H}}=\emptyset$. Let $\alpha=(\sup_{\H\in{\mathbb{H}}}\alpha_{\H})+1<\cuum^{+}$.
It is easy to check that, $\adh\left(\bigwedge_{\H\in{\mathbb{H}}}\H\right)\cap T_{\alpha}=\emptyset$.
Thus, $X$ is injectively core compact at $\infty$.

Since $X$ is injectively core compact at each point, $[X,\$]$ is
diagonal, by Theorem~\ref{thm:diagonaldual}. 
\end{example}

\begin{example}[A $T$-dual convergence space that is not core compact]
\label{exa:tdualnotcorec}  Consider a partition $\left\{ \mathbb{A}_{n}:n\in\omega\right\} $
of the set $\omega^{*}$ of free ultrafilters on $\omega$ satisfying
the condition that for every infinite subset $S$ of $\omega$ and
every $n\in\omega$, there is $\U\in\mathbb{A}_{n}$ with $S\in\U$.
Let $M:=\{m_{n}:n\in\omega\}$ be disjoint from $\omega$ and let
$X:=\omega\cup M$. Define on $X$ the finest convergence in which
$\lim\{m_{n}\}^{\uparrow}=M$ for all $n\in\omega$, and each free
ultrafilter $\U$ on $\omega$ converges to $m_{n}$ (and $m_{n}$
only), where $n$ is defined by $\U\in\mathbb{A}_{n}$. 
\begin{claim*}
$X$ is not core compact.
\begin{proof}
Let $m_{n}\in M$ and $\U\in\mathbb{A}_{n}$. Pick $S\subseteq\omega$,
$S\in\U$, and $k\neq n$. For every $U\in\U$ there is $\W\in\mathbb{A}_{k}$
such that $U\in\W$. But $\lim\W=\{m_{k}\}$ is disjoint from $S$.
\end{proof}
\end{claim*}

\begin{claim*}
$X$ is $T$-core compact, and therefore $[X,\$]$ is pretopological.
\begin{proof}
For each $m_{n}\in M$, the set $M$ is included in every open set
containing $m_{n}$ because $m_{n}\in\bigcap_{k\in\omega}\lim\{m_{k}\}^{\uparrow}.$
If $\U$ is a non-trivial convergent ultrafilter in $X$ then $\lim\U=\{m_{n}\}$
for some $n\in\omega.$ For any $S\in\U$, $S\cap\omega$ is infinite
and any free ultrafilter $\W$ on $S\cap\omega$ belongs to one of
the element $\mathbb{A}_{k}$ of the partition, so that $\lim\W=\{m_{k}\}$
intersects $M$, and therefore any open set containing $m_{n}$.
\end{proof}
\end{claim*}

\begin{claim*}
$[X,\$]$ is diagonal. 
\begin{proof}
Let $\S[\centerdot]:\O_{X}\to\mathbb{F}\O_{X}$ be a selection for
$[X,\$]$ and let $U\in\lim_{[X,\$]}\F$. Now, $\{\bigcap F\colon F\in\F\}$
is a (convergence) cover of $U$.

Let $x\in U$ and $\D$ be a filter on $X$ such that $x\in\lim\D$.
There is an $F\in\F$ and a $D\in\D$ such that $D\subseteq\bigcap F:=V$.

Assume $x\in\omega$, in which case $\D=\{x\}^{\uparrow}$. In particular,
$x\in O$ for every $O\in F$. For every $O\in F$ there is a $T_{O}\in S[O]$
such that $x\in\bigcap T_{O}$. Now, $x\in\bigcap\bigcap_{O\in F}T_{O}\in\S[F]$.
So, $\bigcap\bigcap_{O\in F}T_{O}\in\{x\}^{\uparrow}=\D$.

Assume $x\in M$. In this case, $M\cap O\neq\emptyset$ for all $O\in F$
and, by definition of the convergence on $X$, $M\subseteq O$ for
all $O\in F$. Since $O\in\lim_{[X,\$]}\mathcal{S}[O]$ and $M\subseteq O$,
there is $S\in\mathcal{S}[O]$ such that $x\in\bigcap S$, and, since
each element of $S$ is open, $M\subseteq\bigcap S$. If there is
no $S\in\S[O]$ such that $O\subseteq\bigcap S$ then the filter $\H$
generated by $\{(O\cap\omega)\setminus\bigcap S:S\in\S[O],S\subseteq S_{0}\}$
is non degenerate. Notice that it is not free, for otherwise there
would be an $n\in\omega$ and $\U\in\mathbb{A}_{n}$ with $\U\geq\H$.
But $m_{n}\in\lm\U\cap O$, and there would be $S\in\S[O]$ such that
$\bigcap S\in\U$, which is not possible. Therefore there is $y\in\bigcap_{S\in\backslash S[O]}\left(O\setminus\bigcap S\right)$
which contradicts $O\in\lm_{[X,\$]}\S[O]$. Hence, there is $S_{0}\in\S[O]$
such that $O\subseteq\bigcap S_{0}$. Now, $D\subseteq\bigcap F\subseteq\bigcap_{O\in F}\bigcap S_{O}$.
In particular, $\bigcap_{O\in F}\bigcap S_{O}\in\D$.

Thus, $\{\bigcap J\colon J\in\S[\F]\}$ is a cover of $U$, and $[X,\$]$
is diagonal.
\end{proof}
\end{claim*}
Therefore $[X,\$]$ is pretopological and diagonal, hence topological,
and $X$ is $T$-dual.

\end{example}

\begin{example}[A convergence space that is $T$-core compact but not $T$-dual]
\label{exa:TcorenotTdual} Let $M$ be any topological space such
that $[M,\$]$ is not diagonal and $L$ be a countably infinite set
disjoint from $M$. Let $X=M\cup L$. We now make $X$ into a convergence
space. Define $x\in\lim_{\xi}\{x\}^{\uparrow}$ for each $x\in X$.
For each $x\in M$ define $x\in\lim_{\xi}\F$ whenever $\F$ is the
cofinite filter on $L$ or a filter that that converges to $x$ in
$M$ with its topology. For $x\in L$ define $x\in\lim_{\xi}\{x\}^{\uparrow}$.
Note that any non-empty open subset $O$ of $X$ that intersects $M$
contains all but finitely many points of $L$.
\begin{claim*}
$[X,\$]$ is pretopological. \end{claim*}
\begin{proof}
Let $O\in[X,\$]$ and ${\mathbb{P}}$ be a family of filters on $[X,\$]$
such that $O\in\lim\P$ for every $\P\in{\mathbb{P}}$. Suppose $x\in O\cap L$.
For every $\P\in{\mathbb{P}}$ there is a $P_{\P}\in\P$ such that
$x\in\bigcap P_{\P}$. So, $\bigcap\left(\bigcup_{\P\in{\mathbb{P}}}P_{\P}\right)\in\{x\}^{\uparrow}$.

Suppose that $x\in O\cap M$ and $\F$ is such that $x\in\lim\F$.
We show that for each $\P\in{\mathbb{P}}$, there is a $Q_{\P}\in\P$
such that $O\subseteq Q_{\P}$. Suppose to the contrary that $O\not\subseteq\bigcap P$
for every $P\in\P$. There is a $P_{0}\in\P$ such that $x\in\bigcap P_{0}$.
Consider the filter $\H$ on $L$ generated by the filter-base $\{(O\cap L)\setminus\bigcap P\colon P\in\P\text{ and }P\subseteq P_{0}\}$.
We show that $\H$ is not a free filter. Otherwise, $\H$ is finer
than the cofinite filter on $L$ and therefore $x\in\lim\H$. Since
$\{\bigcap P\colon P\in\P\}$ is a convergence cover of $O$ and $x\in\lim\H$,
there is a $H\in\H$ and a $P\in\P$ such that $H\subseteq\bigcap P$.
Since $O\cap L\setminus\bigcap P\in\H$, we have $\emptyset=H\cap\left(O\cap L\setminus\bigcap P\right)\in\H$.
So, $\H$ is not free and there is $l\in\bigcap\H$. This contradicts
the fact that $\{\bigcap P\colon P\in\P\}$ is in particular a set-theoretic
cover of $O\cap L$. Now $O\subseteq\bigcap\left(\bigcup_{\P\in{\mathbb{P}}}Q_{\P}\right)$,
so that $\bigcap\left(\bigcup_{\P\in{\mathbb{P}}}Q_{\P}\right)\in\F$
because $O$ is open and therefore belongs to $\F$. Hence, $\{\bigcap J\colon J\in\bigwedge_{\P\in{\mathbb{P}}}\P\}$
is a convergence cover of $O$ and $O\in\lim\bigwedge_{\P\in{\mathbb{P}}}\P$.
That is, $X$ is pretopological.\end{proof}
\begin{claim*}
$[X,\$]$ is not diagonal. \end{claim*}
\begin{proof}
Since $[M,\$]$ is not diagonal there is, by Theorem \ref{thm:diagonaldual},
an $x\in M$ and a open neighborhood $J\subseteq M$ of $x$ and a
family $\mathbb{H}$ of filters on $M$ and an injection $\theta\colon{\mathbb{H}}\to\O_{M}(J)$
such that $\adh_{M}\H\cap\theta(\H)=\emptyset$ for each $\H\in\mathbb{H}$
and $x\in\adh_{M}\bigwedge_{\H\in\mathbb{H}}\H$. By our definition
of $\xi$, we have that $x\in\adh_{\xi}\bigwedge_{\H\in\mathbb{H}}\H$.
Let $\U\geq\bigwedge_{\H\in\mathbb{H}}\H$ be an ultrafilter on $M$
such that $x\in\lim_{\xi}\U$.

Define $\theta^{*}\colon{\mathbb{H}}\to\O_{T\xi}(J)$ by $\theta^{*}(\H)=L\cup\theta(\H)$.
Notice that $J\cup L\in\lim_{[X,\$]}\{\theta^{*}(\H)\colon\H\in{\mathbb{H}}\}$.
Since $M$ is closed in $X$, we have $\adh_{M}\H=\adh_{\xi}\H=\adh_{T\xi}\H$
for each $\H\in{\mathbb{H}}$, so that the collection $\{X\setminus\cl_{T\xi}H\colon H\in\H\}$
is an ideal cover of $\theta^{*}(\H)$. Let $\E_{\H}$ be the filter
on $[X,\$]$ that is generated by the collection $\{\O_{\xi}(X\setminus\cl_{T\xi}H)\colon H\in\H\}$.
Then $\theta^{*}(\H)\in\lim_{[X,\$]}\E_{\H}$ for each $\H\in{\mathbb{H}}$,
and the collection $\{\E_{\H}\colon\H\in{\mathbb{H}}\}$ is a selection
for $\{\theta^{*}(\H)\colon\H\in{\mathbb{H}}\}$.

Pick $E_{\H}\in\E_{\H}$ for each $\H\in{\mathbb{H}}$. By definition
of $\E_{\H}$, there is for each $\H$ an $H_{\H}\in\H$ (we may assume
$H_{\H}\subseteq M$) such that $\O_{\xi}(X\setminus\cl_{T\xi}H_{\H})\subseteq E_{\H}$.
So, $\bigcap E_{\H}\subseteq\bigcap\O_{\xi}(X\setminus\cl_{T\xi}H_{\H})=X\setminus\cl_{T\xi}H_{\H}\subseteq X\setminus H_{\H}$.
Now, \[
\bigcap\left(\bigcup_{\H\in{\mathbb{H}}}E_{\H}\right)=\bigcap_{\H\in{\mathbb{H}}}\bigcap E_{\H}\subseteq\bigcap_{\H\in{\mathbb{H}}}X\setminus H_{\H}=X\setminus\left(\bigcup_{\H\in{\mathbb{H}}}H_{\H}\right).\]
 Since $\bigcup_{\H\in{\mathbb{H}}}H_{\H}\in\bigwedge_{\H\in\mathbb{H}}\H$,
$\bigcup_{\H\in{\mathbb{H}}}H_{\H}\in\U$. Hence, $\bigcap\left(\bigcup_{\H\in{\mathbb{H}}}E_{\H}\right)\notin\U$.
Thus, $\left\{ \bigcap E\colon E\in\bigwedge_{\H\in{\mathbb{H}}}\E_{\H}\right\} $
is not a convergence cover of $J\cup L$. Hence, $J\cup L\notin\lim_{[X,\$]}\bigwedge_{\H\in{\mathbb{H}}}\E_{\H}$.
Therefore, $[X,\$]$ is not even $\mathbb{F}_{1}$-diagonal.
\end{proof}
\end{example}

\section{Appendix: convergence spaces}

A family $\A$ of subsets of a set $X$ is called \emph{isotone }if
$B\in\A$ whenever $A\in\A$ and $A\subseteq B$. We denote by $\A^{\uparrow}$
the smallest isotone family containing $\A$, that is, the collection
of subsets of $X$ that contain an element of $\A$. If $\A$ and
$\B$ are two families of subsets of $X$ we say that $\B$ is \emph{finer
than }$\A$, in symbols $\A\leq\B$, if for every $A\in\A$ there
is $B\in\B$ such that $B\subseteq A$. Of course, if $\A$ and $\B$
are isotone, then $\A\leq\B\iff\A\subseteq\B$. This defines a partial
order on isotone families, in particular on the set $\mathbb{F}X$
of filters on $X$. Every family $(\F_{\alpha})_{\alpha\in I}$ of
filters on $X$ admits an infimum \[
\bigwedge_{\alpha\in I}\F_{\alpha}:=\bigcap_{\alpha\in I}\F_{\alpha}=\left\{ \bigcup_{\alpha\in I}F_{\alpha}:F_{\alpha}\in\F_{\alpha}\right\} ^{\uparrow}.\]
On the other hand the supremum even of a pair of filters may fail
to exist. We call \emph{grill of $\A$ }the collection $\A^{\#}:=\{H\subseteq X:\forall A\in\A,\: H\cap A\neq\emptyset\}.$
It is easy to see that $\A=\A^{\#\#}$ if and only if $\A$ is isotone.
In particular $\F=\F^{\#\#}\subseteq\F^{\#}$ if $\F$ is a filter.
We say that two families $\A$ and $\B$ of subsets of $X$ \emph{mesh,
}in symbols $\A\#\B$, if $\A\subseteq\B^{\#},$ equivalently if $\B^{\#}\subseteq\A$.
The supremum of two filters $\F$ and $\G$ exists if and only if
they mesh, in which case $\F\vee\G=\left\{ F\cap G:F\in\F,G\in\G\right\} ^{\uparrow}$.
An infinite family $(\F_{\alpha})_{\alpha\in I}$ of filters has a
supremum $\bigvee_{\alpha\in I}\F_{\alpha}$ if pairwise suprema exist
and for every $\alpha,\beta\in I$ there is $\gamma\in I$ with $\F_{\gamma}\geq\F_{\alpha}\vee\F_{\beta}$.

A \emph{convergence }$\xi$ on a set $X$ is a relation between $X$
and the set $\mathbb{F}X$ of filters on $X$, denoted $x\in\lim_{\xi}\F$
whenever $x$ and $\F$ are in relation, satisfying that $x\in\lim_{\xi}\{x\}^{\uparrow}$
for every $x\in X$, and $\lim_{\xi}\F\subseteq\lim_{\xi}\G$ whenever
$\F\leq\G$. The pair $(X,\xi$) is called a \emph{convergence space.
}A function $f:(X,\xi)\to(Y,\sigma)$ between two convergence space
is \emph{continuous }if \[
x\in\lm_{\xi}\F\then f(x)\in\lm_{\sigma}f(\F),\]
where $f(\F)$ is the filter $\{f(F):F\in\F\}^{\uparrow}$ on $Y$.
If $\xi$ and $\tau$ are two convergences on the same set $X$, we
say that $\xi$ is \emph{finer than }$\tau$, in symbols $\xi\geq\tau$,
if $\lim_{\xi}\F\subseteq\lim_{\tau}\F$ for every $\F\in\mathbb{F}X$.
This defines a partial order on the set of convergence structures
on $X$, which defines a complete lattice for which supremum $\vee_{i\in I}\xi_{i}$
and infimum $\wedge_{i\in I}\xi_{i}$ of a family $\{\xi_{i}:i\in I\}$
of convergences are defined by \[
\lm_{\vee_{i\in I}\xi_{i}}\F=\bigcap_{i\in I}\lm_{\xi_{i}}\F,\]
\[
\lm_{\wedge_{i\in I}\xi_{i}}\F=\bigcup_{i\in I}\lm_{\xi_{i}}\F.\]

Every topology can be identified with a convergence, in which $x\in\lim\F$
if $\F\geq\N(x)$, where $\N(x)$ is the neighborhood filter of $x$
for this topology. A convergence obtained this way is called \emph{topological.
}Moreover, a function $f:X\to Y$ between two topological spaces is
continuous in the usual topological sense if and only if it is continuous
in the sense of convergence. On the other hand, every convergence
determines a topology in the following way: A subset $C$ of a convergence
space $(X,\xi)$ is \emph{closed }if $\lim_{\xi}\F\subseteq C$ for
every filter $\F$ on $X$ with $C\in\F$. A subset $O$ is \emph{open
}if its complement is closed, that is, if $O\in\F$ whenever $\lim_{\xi}\F\cap O\neq\emptyset.$
The collection of open subsets for a convergence $\xi$ is a topology
$T\xi$ on $X$, called \emph{topological modification of $\xi$}.
The topology $T\xi$ is the finest topological convergence coarser
than $\xi.$ $ $If $f:(X,\xi)\to(Y,\tau)$ is continuous, so is $f:(X,T\xi)\to(Y,T\tau)$.
In other words, $T$ is a concrete endofunctor of the category \textbf{Conv
}of convergence spaces and continuous maps. 

Continuity induces canonical notions of subspace convergence, product
convergence, and quotient convergence. Namely, if $f:X\to Y$ and
$Y$ carries a convergence $\tau$, there is the coarsest convergence
on $X$ making $f$ continuous (to $(Y,\tau)$). It is denoted $f^{-}\tau$
and called \emph{initial convergence for $f$ and $\tau.$ }For instance
if $S\subseteq X$ and $(X,\xi)$ is a convergence space, the \emph{induced
convergence by $\xi$ on $S$ }is by definition $i^{-}\xi$ where
$i$ is the inclusion map of $S$ into $X$. Similarly, if $\{(X_{i},\xi_{i}):i\in I\}$
is a family of convergence space, then the product convergence $\Pi_{i\in I}\xi_{i}$
on the cartesian product $\Pi_{i\in I}X_{i}$ is the coarsest convergence
making each projection $p_{j}:\Pi_{i\in I}X_{i}\to X_{j}$ continuous.
In other words, $\Pi_{i\in I}\xi_{i}=\vee_{i\in I}p_{i}^{-}\xi_{i}$.
In the case of a product of two factors $(X,\xi)$ and $(Y,\tau)$,
we write $\xi\times\tau$ for the product convergence on $X\times Y$. 

Dually, if $f:X\to Y$ and $(X,\xi)$ is a convergence space, there
is the finest convergence on $Y$ making $f$ continuous (from $(X,\xi)$).
It is denoted $f\xi$ and called \emph{final convergence for $f$
and $\xi$.} If $f:(X,\xi)\to Y$ is a surjection, the associated
\emph{quotient convergence on $Y$ }is $f\xi$. Note that if $\xi$
is a topology, the quotient topology is not $f\xi$ but $Tf\xi.$ 

The functor $T$ is a reflector. In other words, the subcategory \textbf{Top
}of \textbf{Conv }formed by topological spaces and continuous maps
is closed under initial constructions. Note however that the functor
$T$ does not commute with initial constructions. In particular $T\xi\times T\tau\leq T(\xi\times\tau)$
but the reverse inequality is genrally not true. Similarly, if $i:S\to(X,\xi)$
is an inclusion map, $i^{-}(T\xi)\leq T(i^{-}\xi$) but the reverse
inequality may not hold. A convergence $\xi$ is \emph{pretopological
}or a \emph{pretopology }if $\lim_{\xi}\bigwedge_{\alpha\in I}\F_{\alpha}=\bigcap_{\alpha\in I}\lim_{\xi}\F_{\alpha}$.
Of course, every topology is a pretopology, but not conversely. For
any convergence $\xi$ there is the finest pretopology $P\xi$ coarser
than $\xi$. Moreover, $x\in\lim_{P\xi}\F$ if and only if $\F\geq\V_{\xi}(x)$
where $\V_{\xi}(x):=\bigwedge_{x\in\lm_{\xi}\F}\F$ is called \emph{vicinity
filter of $x$. }The subcategory \textbf{PrTop }of \textbf{Conv }formed
by pretopological spaces and continuous maps is reflective (closed
under initial constructions). Moreover, in contrast with topologies,
the reflector $P$ commutes with subspaces. However, like $T,$ it
does not commute with products. 

The \emph{adherence $\adh_{\xi}\F$ }of a filter $\F$ on a convergence
space $(X,\xi)$ is by definition \[
\adh_{\xi}\F:=\union_{\H\#\F}\lm_{\xi}\H=\union_{\U\in\mathbb{U}(\F)}\lm_{\xi}\U,\]

where $\mathbb{U}X$ denotes the set of ultrafilters on $X$ and $\mathbb{U}(\F)$
denotes the set of ultrafilters on $X$ finer than the filter $\F$.
We write $\adh_{\xi}A$ for $\adh_{\xi}\{A\}^{\uparrow}.$ Note that
in a convergence space $\adh_{\xi}$ may not be idempotent on subsets
of $A$. In fact a pretopology is a topology if and only if $\adh$
is idempotent on subsets. We reserve the notations $\cl$ and $\intr$
to topological closure and interior operators. 

A family $\A$ of subsets of $X$ is \emph{compact at a family $\B$
for $\xi$ if \[
\F\#\A\then\adh_{\xi}\F\#\B.\]
 }We call a family \emph{compact }if it is compact at itself. In particular,
a subset $A$ of $X$ is \emph{compact }if $\{A\}$ is compact, and
\emph{compact at} $B\subseteq X$ if $\{A\}$ is compact at $\{B\}$. 

Given a class $\mathbb{D}$ of filters, a convergence is called \emph{based
in $\mathbb{D}$ }or\emph{ $\mathbb{D}$-based }if for every convergent
filter $\F,$ say $x\in\lim\F$, there is a filter $\D\in\mathbb{D}$
with $\D\leq\F$ and $x\in\lim\D$. A convergence is called \emph{locally
compact }if every convergent filter contains a compact set, and \emph{hereditarily
locally compact }if it is based in filters with a filter-base composed
of compact sets. For every convergence, there is the coarsest locally
compact convergence $K\xi$ that is finer than $\xi$ and the coarsest
hereditarily locally compact convergence $K_{h}\xi$ that is finer
than $\xi$. Both $K$ and $K_{h}$ are concrete endofunctors of \textbf{Conv
}that are also coreflectors. 

If $A\subseteq X$ and $(X,\xi)$ is a convergence space, then $\O(A)$
denotes the collection of open subsets of $X$ that contain $A$ and
if $\A$ is a family of subsets of $X$ then $\O(\A):=\union_{A\in\A}\O(A)$.
A family is called \emph{openly isotone }if $\A=\O(\A)$. Note that
in a topological space $X$, an openly isotone family $\A$ of open
subsets of $X$ is compact\emph{ }if and only if, whenever $\union_{i\in I}O_{i}\in\A$
and each $O_{i}$ is open, there is a finite subset $J$ of $I$ such
that $\union_{i\in J}O_{i}\in\A$. 

If $(X,\xi)$ and $(Y,\sigma)$ are two convergence spaces, $C(X,Y)$
or $C(\xi,\sigma$) deonte the set of continuous maps from $X$ to
$Y$. The coarsest convergence on $C(X,Y)$ making the evaluation
map $e:X\times C(X,Y)\to Y$, $e(x,f)=f(x),$ jointly continuous is
called \emph{continuous convergence }and denoted $[X,Y]$ or $[\xi,\sigma]$.
Explicitely, \[
f\in\lm_{[X,Y]}\F\iff\forall_{x\in X}\forall_{\G\in\mathbb{F}X:x\in\lm_{\xi}\G}\,\, f(x)\in\lm_{\sigma}e\left(\G\times\F\right).\]

\end{document}